\documentclass[12pt]{amsart}
\advance\hoffset by -0.5 truecm
\usepackage{amssymb}
\usepackage{comment}
\usepackage{hyperref}
\newtheorem{theorem}{Theorem}[section]

\newtheorem{definition}[theorem]{Definition}

\newtheorem{lemma}[theorem]{Lemma}

\newtheorem{proposition}[theorem]{Proposition}
\newtheorem{remark}[theorem]{Remark}

\numberwithin{equation}{section}
\theoremstyle{plain}

\usepackage{etoolbox}
\AtEndEnvironment{proof}{\setcounter{claim}{0}}
\newcommand{\R}{\mathbb{R}}

\newcommand{\Da}{\Delta}
\newcommand{\ve}{\varepsilon}

\begin{document}
\title[Solutions of the Yamabe equation]{Multiplicity of 2-nodal solutions \\ of the  Yamabe equation}

\author{Jorge D\'avila}
\email{jorge.davila@cimat.mx}

\author{H\'ector Barrantes G.}
\email{hector.barrantes@ucr.ac.cr}%

\author{ Isidro H. Munive }
\email{isidro.munive@academicos.udg.mx}

\begin{abstract} Given any closed Riemannian manifold $(M,g)$, we use the gradient flow  method and Sign-Changing Critical Point Theory to
prove multiplicity results for 2-nodal solutions   of a subcritical non-linear equation on $(M,g)$, see Eq. \eqref{yam-nodal-1} below. If $(N,h)$ is a closed Riemannian manifold of constant  positive scalar curvature our result gives multiplicity results for the  Yamabe-type equation on the
Riemannian product $(M\times N , g + \ve h )$, for $\ve >0$ small.

\end{abstract}
\maketitle

\section{\textbf{Introduction}}


On a compact  Riemannian manifold $(M^n , g)$ without boundary of dimension $n\geq 3$, we consider the following equation 
\begin{equation} \label{yam-nodal-1}
-\ve^2 \Da_gu+\left(\frac{s_g}{a_{m+n}}\varepsilon^{2}+1\right)u=\vert u\vert ^{p_{m+n}-2}u,
\end{equation}
where $s_g$ is the scalar curvature of $g$,  $ \Delta_g$  is the is the Laplace Beltrami operator associated to $g $,  $a_{m+n} =\frac{4(n+m-1)}{n+m-2}$, $p_{n+m} =\frac{2(n+m)}{n+m-2}$, with $m \in \mathbb{N}$. Moreover, we  consider $\varepsilon>0$ small enough so that 
\begin{equation}
\label{coereq}
1+\frac{s_g}{a_{m+n}}\varepsilon^{2}>c_{\varepsilon}\quad \text{in $M$},
\end{equation}
 for some $c_{\varepsilon}>0$. 
 
 The  study of this  equation is motivated, on one hand, by the Yamabe problem on products of Riemannian manifolds. If  $u:M\rightarrow \R$ is a positive solution of Eq. \eqref{yam-nodal-1} then $u$ solves the Yamabe equation in the product $(M^n\times N^m, g+\varepsilon^2h)$, where $(N^m,h)$ is a Riemannian manifold with constant scalar curvature $s_h$ equal to $a_{m+n}$, see, for instance, \cite{Jimmy} for details.

There has also been interest in {\it nodal} solutions of non-linear equations of the type \eqref{yam-nodal-1} (i.e. solutions that change sign). See for instance the articles
\cite{Ammann2, Clapp1, Clapp2, Ding, JC, GHenry, Robert} and, more recently, the paper \cite{PrV}.  Nodal solutions $u$ of \eqref{yam-nodal-1} do not give metrics of constant scalar curvature since $u$ vanishes at some points and therefore $|u|^{p_n -2} g$ is
not
a Riemannian metric, but they might have geometric interest. The existence of at least one nodal solution is proved
in general cases in \cite{Ammann2}, as minimizers for the second Yamabe invariant. But there are not as many results
about multiplicity of nodal solutions as in the positive case.


In \cite{CM}, M.Clapp and M. Micheletti considered the problem of obtain $2$-nodal solutions
to the equation
\[
-\ve^2 \Da_gu+ u=\vert u\vert ^{p_{m+n}-2}u,
 \]
 over a closed Riemannian manifold $(M,h)$.  In order to study this problem, they used gradient flow techniques to prove the existence of $2$-nodal solutions. In this work we obtain existence results for Eq. \eqref{yam-nodal-1}, see Theorem \ref{Teo1-Existence}, using  gradient flow  techniques from \cite{CM} (see Chapter 1 of \cite{Zou} for details)
for the functional
 \begin{equation}\label{Jve-1}
J_{\ve} (u) \doteq
\frac{1}{\ve^{n}} \int_M \left( \frac{1}{2} \ve^2 | \nabla_g u |_g^2  +  \frac{1}{2}\left( \frac{s_g}{a_{m+n}}\ve^2 +1\right) u^2 -\frac{1}{p_{m+n}}  \vert u\vert^{p_{m+n}} \right) d\mu_g  .
\end{equation}
We recall here  that \eqref{yam-nodal-1} is the Euler-Lagrange equation of $J_{\ve}$.  The Nehari manifold $\mathcal{N}_\ve$ associate to the functional $J_{\ve}$ is the following set:
\[
\mathcal{N}_{\ve} \doteq   \Big\{ u\in H_{\ve}\setminus\{0\} :\mathcal{L}_{\ve}(u,u) =|u|^{p}_{p,\ve}  \Big\},
\]
where   $|u|_{p,\ve}$ and $\mathcal{L}_{\ve}(u,u)$  are given by \eqref{LpSpe} and  \eqref{Lep}, respectively. Notice that any sign changing solution   belongs to the set
\begin{equation}
\label{Eep}
\mathcal{E}_{\ve}\doteq\Big\{u\in H_{\ve}: u^{+}, u^{-}\in \mathcal{N}_{\ve}\Big\}\subset \mathcal{N}_\ve.
\end{equation}
Our first main result is the following.

\begin{theorem} \label{Teo1-Existence}
 (Existence)
 For every $\ve >0$ there exists $u_{\ve}\in \mathcal{E}_{\ve}$  such that $J_{\ve}(u_{\ve})=\mathbf{d}_\ve$, where $\mathbf{d}_\ve \doteq  \inf_{\mathcal{E}_{\ve}}J_{\ve}$,
  and $u_{\ve}$ is a sign changing solutions of Eq. \eqref{yam-nodal-1}. Moreover,
$\textbf{d}_\ve \geq 2 \textbf{m}_\ve$, where $\textbf{m}_{\ve} \doteq \inf_{\mathcal{N}_{\ve}}J_{\ve}$.
 \end{theorem}

We consider the equivariant Lusternik-Schnirelmann category $\mathrm{Cat}_G(X)$ of a $G$-space $X$ is the
smallest integer $k$ such that $X$ can be covered by $k$ locally closed $G$-invariant subsets $X_1,\ldots, X_k$, see Definition 5.5 in  \cite{BMW}.

\begin{theorem}
\label{Teo-Mult-Cat}
There exists $\varepsilon_0 > 0$ such that for any $\varepsilon\in (0, \varepsilon_0)$, $J_{\varepsilon}$ has at least $\mathrm{Cat}(\mathcal{Z}_{\varepsilon} \cap J^{d_{\varepsilon}+\delta_0})$ critical points.  Moreover  $\mathrm{Cat}(\mathcal{Z}_\ve \cap J_{\ve}^{d_{\ve} + \delta_0}  )\geq \mathrm{Cupl}\ (\mathcal{Z}_{\ve}\cap J_{\ve}^{d_{\ve}+ \delta_0})\geq Cupl(M)$.
\end{theorem}

 We also obtain a multiplicity result, see  Theorem \ref{Teo2-Multip} below, with the help of the \emph{center of mass} of a function introduced by  Petean in the paper \cite{Jimmy}. This \emph{center of mass} plays the role of the \emph{barycenter map}, see for instance \cite{BMW}, in the Riemannian setting. Given the set 
 $F(M) \doteq \{(x,y)\in M \times M: x\neq y \}$, we consider the quotient space $C(M)$ of $F(M)$, under the free action $\theta(x,y)=(y,x)$, and define $\mathcal{H}^{*}$ for singular cohomology with coefficients in $\mathbb{Z}_{2}$. Recall that the \emph{cup-length} of a topological space $X$, denoted  by $\text{cupl}\, X$, is the smallest integer $k\geq 1$ such that  the cup-product of any $k$ cohomology classes in $\widetilde{\mathcal{H}}^{*}(X)$ is zero, where  $\widetilde{\mathcal{H}}^{*}$ is reduced cohomology.
 
 \begin{theorem}
\label{Teo2-Multip} (Multiplicity) There exists $\ve_{0}>0$  such that for any $\ve \in (0, \ve_0)$, problem \eqref{yam-nodal-1}  has at least $\mathrm{cupl}\, C(M)$ pairs of sign solutions $\pm u$ with $J_{\ve}(u)< d_{\ve} + k_{0}$.
\end{theorem}

 In \cite{CMar} it is proved that
 \[
 \text{cupl}\, C(M)\geq n+1,
 \]
 and
 
\begin{theorem}
\label{cuplCM} If $\mathcal{H}^{i}(M)=0$ for all $0<i <m$ and if there are $k$ cohomology classes $\xi_1,\ldots,\xi_k \in \mathcal{H}^{m}(M)$ whose cup-product is non-trivial, then
 \[
 \mathrm{cupl}\,C(M)\geq k+n.
 \]
 \end{theorem}

From Theorem \ref{cuplCM} we get that if $M=\mathbb{S}^{1}\times\cdots\times \mathbb{S}^{1}$ $n-$times, then $\text{cupl}\, C(M)=2n$.

\section{\textbf{Preliminaries}}

  Let $H_{\varepsilon}$ be the Hilbert space $H^1_g(M)$ equipped with the inner product
\begin{equation} \label{inner-prod-1}
\langle u, v\rangle_{\varepsilon} \doteq \frac{1}{\ve^n}\int_M \left(\ve^2 \left\langle \nabla_g u,   \nabla_g v  \right\rangle_g + uv \right) d \mu_g
\end{equation}
 and the induced  norm
\[
\|u\|^2_{\ve} \doteq \frac{1}{\ve^n}\int_M  \left(\ve^2 \left| \nabla_gu  \right|_g^2+   u^2  \right) d\mu_g.
\]
Consider the bilinear form  $\mathcal{L}_{\ve}:H_{\ve}\times H_{\ve}\rightarrow \R$ given by

\begin{equation}
\label{Lep}
\mathcal{L}_{\ve}(u,v)\: \doteq\frac{1}{\ve^n}\int_M \Bigg[\ve^2 \left\langle \nabla_g u,   \nabla_g v  \right\rangle_g + \left(\frac{s_g}{a_{m+n}}\ve^2 +1 \right)u v \Bigg] d \mu_g,\quad u,v\in H_{\ve}.
\end{equation}
 From \eqref{coereq} we have $\mathcal{L}_{\ve}$ is coercive, meaning that
\begin{equation}
\label{Lcoer}
c_{\varepsilon}\|u\|_{\ve}\leq \mathcal{L}_{\ve}(u,u)^{\frac{1}{2}}\leq c^{-1}_{\varepsilon}\|u\|_{\ve},\quad \forall\ u\in H_{\ve},
\end{equation}
 for some $c_{\varepsilon}>0$.   This implies that  $\mathcal{L}_{\ve}(\cdot,\cdot)$ and $\langle \cdot, \cdot\rangle_{\varepsilon}$ are equivalent inner products in $H_{\varepsilon}$. For simplicity, we set $\mathcal{L}_{\ve}(u)   \doteq \mathcal{L}_{\ve}(u,u)$,  $u\in H_{\ve}$.
 
Let $L^q_{\ve}$  be the Banach spaces $L^q_g(M)$ with the norm
\begin{equation}
\label{LpSpe}
|u|_{q,\ve} \doteq \left(\frac{1}{\ve^n}\int_M|u|^qd\mu_g\right)^{\frac{1}{q}}.
\end{equation}

 For $q\in (2,p_n )$ if $n\geq 3$ or $q > 2$ if $n=2$,  the embedding $i_{\ve} : H_{\ve} \hookrightarrow L^q_{\ve}$ is a continuous map. Moreover, one can easily check that there exists a  constant $c$  independent of $\ve$ such that
\[
|i_{\ve} (u) |_{q,\ve}\leq c\|u\|_{\ve},\quad \text{for any $u\in H_{\ve}$}.
\]
 Let $q' \doteq\frac{q}{q-1}$ so that $\frac{1}{q} + \frac{1}{q'} =1$. Notice that for  $v\in L^{q'}_{\ve}$, the map
\[
\varphi \rightarrow \langle v,i_{\ve}\left(\varphi\right)\rangle \doteq \frac{1}{\ve^n}\int_Mv\cdot i_{\ve}\left(\varphi\right)d\mu_g,\quad \varphi \in H_{\ve},
\]
 is a continuous functional by the compact embedding $i_{\ve}:H_{\ve}\hookrightarrow L^q_{\ve}$. For each $\varphi\in L^{q'}_{\ve}$, define the functional $\mathcal{F}_{\varphi}: H_{\ve}\rightarrow\R$ by
\[
\mathcal{F}_{\varphi}(v)\doteq\frac{1}{\ve^n}\int_M\varphi \cdot i_{\ve}\left(v\right)  d\mu_g ,\quad\forall\ v \in H_{\ve}.
\]
By the Lax-Milgram Theorem, there exists  $u\in H_{\ve}$ such that $\mathcal{L}_{\ve}(u,v)= \mathcal{F}_{\varphi}(v)$  for all $v\in H_{\ve}$. In other words, such function $u\in H_{\ve}$ is the weak solution of
\begin{equation}\label{eqv}
-\ve^2 \Da_gu+\left(\frac{s_g}{a_{m+n}}\varepsilon^{2}+1\right)u=  \varphi \quad\text{in } M,
\end{equation}
 where $\varphi\in L^{q'}_{\ve}$. Recall that, by elliptic regularity theory,   if $\varphi \in C^{k,\alpha} (M)$, then $u\in C^{k+2,\alpha} (M)$.

From now on we use the notation $ p\doteq p_{m+n}$. Consider the functional $J_{\ve} \colon H_{\ve} \to  \R$ given by
\[
J_{\ve} (u) \doteq
\frac{1}{\ve^{n}} \int_M \left( \frac{1}{2} \ve^2 | \nabla_g u |_g^2  +  \frac{1}{2}\left( \frac{s_g}{a_{m+n}}\ve^2 +1\right) u^2 -\frac{1}{p}  | u|^p \right) d\mu_g.
\] Its gradient is given by $\nabla J_{\ve } \colon H_{\ve } \to L(H_{\ve }, \R)$, where
\begin{eqnarray*}
 \nabla J_{\ve}(u)(v) &\doteq & \frac{1}{\ve^{n}} \int_M \left(  \varepsilon^2 \left\langle \nabla_g u,   \nabla_g v  \right\rangle_g  +
\left( \frac{s_g}{a_{m+n}}\ve^2 +1\right) uv  - | u|^{p-2}uv \right) d\mu_g\\
&=& \mathcal{L}_{\ve}(u,v)-  \frac{1}{\ve^{n}} \int_M   | u|^{p-2}uv \;  d\mu_g
\end{eqnarray*}

Consider the operator $J'_{\ve } \colon H_{\ve } \to H_{\ve } $ given by
$J'_{\ve}(u)\doteq u-K_{\ve}(u)$,  where   $K_{\ve}(u)$ is the solution of \eqref{eqv} with $\varphi = | u|^{p-2}u$ and $u\in H_{\ve}$. Then, 
\begin{equation}
\label{grdJe}
\nabla J_{\ve}(u)(v) =\mathcal{L}_{\ve}(J'_{\ve}(u),v), \quad \text{for $u,v\in H_{\ve}$}.
\end{equation}

 The Nehari manifold $\mathcal{N}_\ve$ associate to the functional $J_{\ve}$ is the following set:
\[
\mathcal{N}_{\ve} \doteq\Big\{ u\in H_{\ve}\setminus\{0\} :\mathcal{L}_{\ve}(u,u) =\bigl|u \bigr|^{p}_{p,\ve}  \Big\}.
\]

\begin{lemma}
\label{JP-S}
The functional $J_{\varepsilon}: H_{\varepsilon}\rightarrow \R$ satisfies the Palais-Smale condition. Moreover, the functional $J_{\ve}$ restricted to $\mathcal{N}_{\ve}$  is coercive.
\end{lemma}
\begin{proof}
Suppose that $(u_k)\subset H_{\varepsilon}$, with 
\begin{equation}
\label{Jvebdd}
(J_{\varepsilon}(u_k))\quad \text{bounded},
\end{equation}
and 
\begin{equation}
\label{Jprime}
J'_{\varepsilon}(u_k)\rightarrow 0\quad \text{in $H_{\varepsilon}$}.
\end{equation}
Recall that \eqref{Jprime} means that
\begin{equation}
\label{JT}
u_k-K(u_k)\rightarrow 0 \quad \text{in $H_{\varepsilon}$}.
\end{equation}
Hence, for every $\delta>0$ we have,
\begin{eqnarray*}
|\mathcal{L}_{\ve}(J'_{\varepsilon}(u_k),v)|&=&\Big|\mathcal{L}_{\ve}(u_k,v)-\frac{1}{\ve^{n}} \int_M   |u_k|^{p-2}u_kv \;  d\mu_g\Big| <\delta \mathcal{L}_{\ve}(v)^{\frac{1}{2}},
\end{eqnarray*}
for $k>0$ large enough and for every $v\in H_{\ve}$. If we take $v=u_k$ above we find 
\[
\Big|\mathcal{L}_{\ve}(u_k)-\frac{1}{\ve^{n}} \int_M   |u_k|^{p} \;  d\mu_g\Big| <\delta \mathcal{L}_{\ve}(u_k)^{\frac{1}{2}},
\]
for every $\delta>0$, and $k>0$ large enough. In particular, for $\delta=1$, 
\begin{equation}
\label{ukpdel}
|u|_{p,\ve}^{p}\leq \mathcal{L}_{\ve}(u_k)+\mathcal{L}_{\ve}(u_k)^{\frac{1}{2}},
\end{equation} 
for $k>0$ sufficiently large. Since \eqref{Jvebdd} says that
\[
\frac{1}{2}\mathcal{L}_{\ve}(u_k)-\frac{1}{p}|u|_{p,\ve}^{p}<C<\infty,
\]
for all $k$ and some constant $C>0$, we deduce from \eqref{ukpdel} that
\begin{align*}
\mathcal{L}_{\ve}(u_k)\leq 2C+\frac{2}{p}\left(\mathcal{L}_{\ve}(u_k)+\mathcal{L}_{\ve}(u_k)^{\frac{1}{2}}\right).
\end{align*} 

Given that $\mathcal{L}_{\ve}$ is coercive, see \eqref{Lcoer}, and that $2/p<1$, we get that $(u_k)$ is bounded in $H_{\ve}$. Hence, there exists a subsequence $(u_{k_j})$ and $u\in H_{\ve}$, with $u_{k_j}\rightharpoonup u$ weakly in $H_{\ve}$, and $u_{k_j}\rightarrow u$ in $L^{p}_{\ve}$  by the compact embedding $H_g^{1}(M)\hookrightarrow L_g^{p}(M)$. From this we get that $|u_{k_j}|^{p-2}u_{k_j} \rightarrow |u|^{p-2}u$ in $L^{p'}_{\ve}$. Therefore, $K(u_{k_j})\rightarrow K(u)$ in $H_{\ve}$. So, \eqref{JT} implies
\[
u_{k_j}\rightarrow u \quad \text{in $H_{\ve}$}.
\]

We now prove that $J_{\ve}$ restricted to $\mathcal{N}_{\ve}$  is coercive. By definition,
\[ J_{\ve}(u)=\frac{1}{2}\mathcal{L}_{\varepsilon}(u) -\dfrac{1}{p} |u|_{p ,\ve}^{p}.\]
Now, if $u\in \mathcal{N}_{\ve}$, we have $\mathcal{L}_{\ve}(u)=\bigl|u \bigr|^{p}_{p,\ve} $. So,
\[
 J_{\ve}(u)=\frac{1}{2}\mathcal{L}_{\ve}(u) -\dfrac{1}{p}\mathcal{L}_{\ve}(u)
= \Big(\frac{1}{2}-\dfrac{1}{p}\Big)\mathcal{L}_{\ve}(u)\geq \dfrac{p-2}{2p} c_{\varepsilon}\|u \|_{\ve}^2.
\]
Here we have used again that $\mathcal{L}_{\ve}$ is coercive.
\end{proof}
 
Now, if we define
\begin{equation}
\label{Sve}
S_{\ve} \doteq\inf \left\{ \dfrac{\mathcal{L}_{\ve}(u)}{|u|^{2}_{q,\ve}}:\quad u\in H_{\ve}, u\neq 0 \right\},
\end{equation}
we  get that
\begin{equation}
\label{mveSve}
\textbf{m}_{\ve}=\frac{p-2}{2p}S_{\ve}^{\frac{p}{p-2}},
\end{equation}
where  $\textbf{m}_{\ve} \doteq \inf_{\mathcal{N}_{\ve}}J_{\ve}$. Identity \eqref{mveSve}  follows from the fact that if  $u\in H_{\ve}\setminus\{0\}$, then $t_{\ve}(u)u\in \mathcal{N}_{\ve}$, where
\begin{equation}
\label{tnehari}
t^{p-2}_{\ve}(u) \doteq  \frac{\mathcal{L}_{\ve}(u)}{|u|^p_{p,\ve}}.
\end{equation}

  We close this section with the following result from \cite{Jimmy}. It is well known that  there exists a unique (up to translation) positive
finite-energy solution $U$  of the equation 
\begin{equation}
\label{limeq}
-\Delta U + U = \vert U\vert^{q-2}U \quad \text{on $\R^n$}.
\end{equation}

Moreover, the function $U$ is radial around some chosen point,  and it is exponentially decreasing at infinity
(see \cite{Gidas}):
\[
|U (x) |  \leq C e^{-c| x | },
\]
and
\[
|\nabla U (x) | \leq C  e^{-c| x | }.
\]
Consider  the functional $E: H^1 (\R^n ) \rightarrow \R$,
\[
E(f)\doteq \int_{\R^n} \left(\frac{1} {2} \| \nabla f  \|^2 + \frac{1} {2} f^2 -\frac{1}{q} \vert f\vert ^q  \right) dx,
\]
and the corresponding Nehari Manifold
\[
N(E) \doteq\left\{ u\in H^{1}(\mathbb{R}) : \int_{\R^n}\big( \| \nabla u  \|^2 +  u^2  \big)dx = \dfrac{1}{q}\int_{\R^n} | u|^q dx \right\}.
 \]
Note that $U$ is a critical point of $E$ and minimizer of the functional $E$ restricted  to $N(E)$. The minimum is then
\begin{equation}
\label{mE}
 \textbf{m}(E)\doteq\min\left\{ E(u): u\in N(E)\right\}=\dfrac{q-2}{2q}\|U\|_{q}^{q}.
\end{equation}

\begin{theorem}
\label{mep}
We have that   $\lim_{\ve\rightarrow 0} \textbf{m}_\ve=\textbf{m}(E)$, where $\textbf{m}(E)$ is given by \eqref{mE}.
\end{theorem}

\section{\textbf{Existence Of Nodal Solutions}}

 Recall that for $u\in H_{\ve}$,  $J'_{\ve}(u) \doteq  u-K_{\ve}(u)$,  where   $K_{\ve}(u)$ is the solution of \eqref{eqv} with $\varphi = | u|^{p-2}u$, is the gradient of $J_{\ve}$ with respect to
 the inner product  $\mathcal{L}_{\varepsilon}(\cdot, \cdot )$. Consider the  negative gradient flow
 $\varphi_{\ve}: \mathcal{G_{\ve}}\to H_{\ve} $  defined by
\[
\begin{cases}
\dfrac{d}{dt}\varphi_{\ve}(t,u) =  - J'_{\ve}(\varphi_{\ve}(t,u)), \\
\varphi_{\ve}(0,u)=u,
\end{cases}
\]
where $\mathcal{G_{\ve}} = \{ (t, u) : u \in H_g^1(M), \; 0 \leq t \leq T^{\ve}(u) \} $ and $T^{\ve}(u) \in (0, + \infty)$ is  the maximal existence time for $\varphi_{\ve}$.

\begin{definition}
A set $\mathcal{D} \subset  H_g^1(M)$ is \emph{strictly positively invariant} under the  flow $\varphi_{\ve}$,  if
for every $u \in \mathcal{D}   $ and $t \in  (0, T^{\ve}(u))$,  $\varphi_{\ve} (t, u) \in \overset{\circ}{\mathcal{D}}$, where
$\overset{\circ}{\mathcal{D}}$ denotes the interior of $\mathcal{D}$ in   $H_{\ve}$.
\end{definition}

If $\mathcal{D}$ is strictly positively invariant under the flow $\varphi_{\ve}$, the set
\[ 
\mathcal{A}_\ve(\mathcal{D}) \doteq \{u\in H_{g}^{1}(M) : \varphi_{\ve}(t,u)\in\mathcal{D} \text { for some } t\in(0, T^{\ve}(u)) \}
\]
is an open subset of $H^1_g(M)$. We define the  convex cone of non-negative  functions by
$\mathcal{P}  \doteq \{u\in H_{\ve}: u\geq 0\}$.
For $\alpha > 0$ define also the tubular neighborhood
\[
\mathcal{B}_{\alpha}(\ve, \pm \mathcal{P}) \doteq\Big\{u\in H_{\ve}: \text{dist}_{\ve}(u,\pm \mathcal{P})\leq \alpha \Big\},
\]
where
\[
\text{dist}_{\ve}(u,\pm \mathcal{P}) \doteq \min_{v \in \pm\mathcal{P}}\mathcal{L}_{\ve}(u-v,u-v)^{\frac{1}{2}}.
\]
 For $a\in \mathbb{R}$, we consider the set
$J_{\ve}^{a}\doteq J_{\ve}^{-1}((-\infty, a])=\{u\in H_{\ve}: J_{\ve}(u)\leq a\}$.
 Moreover, for $\ve > 0$ we let
\[
\mathcal{D}_{\ve} \doteq \mathcal{B}_{\alpha}(\ve, \mathcal{P})\cup \mathcal{B}_{\alpha}(\ve, -\mathcal{P})\cup J_{\ve}^{0},
\]
and
\begin{equation}
\label{defZep}
\mathcal{Z}_{\ve}\doteq  H_{\ve}\setminus \mathcal{A}_{\ve}(\mathcal{D}_{\ve}).
\end{equation}

Our  first result is the following lemma. 

\begin{lemma}
\label{Bspi}
If $\alpha \doteq \frac{1}{2}S_{\ve}^{p/2(p-2)} $,  then
\begin{enumerate}
    \item $\left(\mathcal{B}_{\alpha}(\ve, \mathcal{P}\right)\cup \mathcal{B}_{\alpha}(\ve, -\mathcal{P}))\cap \mathcal{E}_{\ve}=\emptyset$;
    \item $ {B}_{\alpha}(\ve, \pm \mathcal{P})$ is strictly positive invariant for the flow $\varphi_{\ve}.$
\end{enumerate}

\end{lemma}
\begin{proof}

\begin{enumerate}
\item  First, note that
\begin{eqnarray}
\label{umdis}
\vert u^{-}\vert_{p,\ve}=\min_{v\in \mathcal{P}}\vert u-v\vert_{p,\ve}&\leq & S_{\ve}^{-1/2}\min_{v\in\mathcal{P}}\mathcal{L}_{\ve}(u-v,u-v)^{\frac{1}{2}}\\
\nonumber
&=& S_{\ve}^{-1/2} \text{dist}_{\ve}(u, \mathcal{P}).
\end{eqnarray}
Then, if $u\in \mathcal{E}_{\ve}\cap \mathcal{B}_{\alpha}(\ve, \mathcal{P})$, 
\[
0< S^{p/p-2}_{\ve}\leq \mathcal{L}_{\ve}(u^-)=|u^-|^p_{p,\ve}\leq S_{\ve}^{-p/2} \text{dist}_{\ve}(u, \mathcal{P})^p\leq \frac{1}{2^p}S^{p/p-2}_{\ve}.
\]
This contradiction gives us that $\mathcal{B}_{\alpha}(\ve, \mathcal{P})\cap \mathcal{E}_{\ve}=\emptyset$. In similar fashion, $\mathcal{B}_{\alpha}(\ve, -\mathcal{P})\cap \mathcal{E}_{\ve}=\emptyset$. Hence, (1) is established.

\item   We prove the assertion for  $\mathcal{B}_{\alpha}(\ve,  \mathcal{P}) $. We first show that
if  $u\in \mathcal{B}_{\alpha}(\ve, \pm \mathcal{P}) $, then $K_{\ve}(u)$  is in the interior of
$\mathcal{B}_{\alpha}(\ve, \pm \mathcal{P})$. 
Observe that
\begin{eqnarray*}
  \text{dist}_{\ve}(K_{\ve}(u), \mathcal{P}) \mathcal{L}_{\ve}(K_{\ve}(u)^{-})^{\frac{1}{2}}  
&\leq & \mathcal{L}_{\ve}(K_{\ve}(u)^{-},K_{\ve}(u)^{-}) \\
&=& \frac{1}{\ve^{n}} \int_M   | u|^{p-2}u  K_{\ve}(u)^-\;  d\mu_g \;\;\; \\
&\leq& |u^-|_p^{p-1} \left| K_{\ve}(u)^{-}\right|_p\\
&\leq& S_{\ve}^{-p/2} \text{dist}_{\ve}(u, \mathcal{P})^{p-1}\mathcal{L}_{\ve}(K_{\ve}(u)^{-})^{\frac{1}{2}}\quad \text{(by \eqref{umdis})}\\
&\leq& S_{\ve}^{-p/2} \left(\frac{1}{2}S_{\ve}^{p/2(p-2)}\right)^{p-1}\mathcal{L}_{\ve}(K_{\ve}(u)^{-})^{\frac{1}{2}}\\
&=& \frac{1}{2^{p-1}} S_{\ve}^{p/2(p-2)} \mathcal{L}_{\ve}(K_{\ve}(u)^{-})^{\frac{1}{2}}.
\end{eqnarray*}
Hence,
\[
\text{dist}_{\ve}(K_{\ve}(u), \mathcal{P}) \leq  \frac{1}{2^{p-1}} S_{\ve}^{p/2(p-2)}.
\]
 It follows that $K_{\ve}(u)$ is in the interior of $\mathcal{B}_{\alpha}(\ve, \mathcal{P})$. Given that the set $\mathcal{B}_{\alpha}(\ve, \mathcal{P})$ is convex, we get that 
 \begin{equation}
 \label{ulmK}
 u-\lambda (J'_{\ve}(u))=(1-\lambda)u-\lambda K_{\ve}(u)\in \mathcal{B}_{\alpha}(\ve, \mathcal{P})
\end{equation}  
for all $u\in \mathcal{B}_{\alpha}(\ve, \mathcal{P})$ and $\lambda\in [0,1]$. Then, we get from \eqref{ulmK} that
\begin{equation}
\label{distuB}
\lim_{\lambda\rightarrow 0^+}\frac{\mathrm{dist}\,\left(u+\lambda (-J'_{\ve}(u))), \mathcal{B}_{\alpha}(\ve, \mathcal{P})\right)}{\lambda}=0, \quad \text{for every $u\in \mathcal{B}_{\alpha}(\ve, \mathcal{P})$}.
\end{equation} 
Hence, using \eqref{distuB}, we get from Theorem 1.49 in \cite{Zou} that
\begin{equation}
\label{phiBeP}
\varphi_{\varepsilon}(u,t)\in \mathcal{B}_{\alpha}(\ve, \mathcal{P}),\quad \text{for every $u\in \mathcal{B}_{\alpha}(\ve, \mathcal{P}),\, 0\leq t<T_{\ve}(u)$}.
\end{equation}
Finally, using a convexity-type argument as in Proposition 3.1 in \cite{BMW}, we get from \eqref{phiBeP} that $\varphi_{\varepsilon}(u,t)\in \mathrm{int}\,\mathcal{B}_{\alpha}(\ve, \mathcal{P})$,  for every $u\in \mathcal{B}_{\alpha}(\ve, \mathcal{P})$ and $ 0< t<T_{\ve}(u)$.
\end{enumerate}
\end{proof}

\begin{remark} We have that  $\inf_{\mathcal{E}_{\ve}}J_{\ve}$ is  attained and any minimizer of $J_{\ve}$ on $\mathcal{E}_{\ve}$ is a sign changing solution to Eq. \eqref{yam-nodal-1}. Hence, we set
\begin{equation}
\label{depsilon}
\mathbf{d}_{\ve} \doteq  \inf_{\mathcal{E}_{\ve}}J_{\ve}.
\end{equation}
By Lemma \ref{Bspi}, we have that $\mathcal{D}_{\ve}$ is strictly positive invariant for
the flow  $\varphi_{\ve}$. Therefore, the set $\mathcal{Z}_{\ve}$ is a closed subset of $H_{\ve}$. Moreover, every function in $\mathcal{Z}_{\ve}$ is sign changing and every sign changing solution
for Eq. \eqref{yam-nodal-1}  lies in $\mathcal{Z}_{\ve}$. Therefore, 
\[
\mathbf{d}_{\ve}\geq \inf_{\mathcal{Z}_{\ve}}J_{\ve}. 
\]
\end{remark}

\begin{lemma}[Ekeland's variational principle]
\label{ekeland}
  Given $\ve >0$, $\delta>0$ and $u\in \mathcal{Z}_\ve$ such that 
  $J_\ve(u)\leq \inf_{\mathcal{Z}_{\ve}}J_{\ve}+ \delta$, there exists $v\in \mathcal{Z}_\ve$ such that $J_\ve (v)\leq J_\ve (u)$, $\mathcal{L}_\ve(u-v)^{\frac{1}{2}} \leq \sqrt{\delta}$ and $\mathcal{L}_\ve(J'_{\ve}(v))^{\frac{1}{2}} \leq \sqrt{\delta}$.
\end{lemma}
\begin{proof}
  Let $t_0 \doteq \inf\Big\{ t>0:  \sqrt{\delta}\leq \mathcal{L}_\ve(\varphi_{\ve}(t,u)-\varphi_{\ve}(0,u))^{\frac{1}{2}} \Big\} \in (0,\infty] $.  Suppose  that  $\sqrt{\delta} < \mathcal{L}_\ve( J'_{\ve}(\varphi_{\ve}(t,u)))^{\frac{1}{2}} \quad \text { for all } t\in (0,t_0).$ This implies,
    \[ 
\mathcal{L}_\ve( J'_{\ve}(\varphi_{\ve}(t,u)))^{\frac{1}{2}}  \leq \dfrac{1}{\sqrt{\delta}}\mathcal{L}_\ve( J'_{\ve}(\varphi_{\ve}(t,u)))  \text { for all } t\in (0,t_0).
    \]
Hence,
\begin{align*}
\sqrt{\delta}=&\mathcal{L}_\ve(\varphi_{\ve}(t_0,u)-\varphi_\ve(0,u))^{\frac{1}{2}}=\mathcal{L}_\ve\left(\int_{0}^{t_0}\dfrac{d}{dt}\varphi_{\ve}(t,u)dt\right)^{\frac{1}{2}}\\
    =&\mathcal{L}_{\ve}\left( \int_{0}^{t_0} - J'_{\ve}(\varphi_{\ve}(t,u))\right)^{\frac{1}{2}}\leq \int_{t_0}^{0}\mathcal{L}_{\ve}(J'_{\ve}(\varphi_{\ve}(t,u)))^{\frac{1}{2}}dt\\
    \leq & \dfrac{1}{\sqrt{\delta}}\int_{t_0}^{0}\mathcal{L}_{\ve}(J'_{\ve}(\varphi_{\ve}(t,u)))dt=\dfrac{1}{\sqrt{\delta}}\int_{t_0}^{0}\dfrac{d}{dt}J_{\ve}(\varphi_{\ve}(t,u))dt\\
    =& \dfrac{1}{\sqrt{\delta}}(J_{\ve}(u)-J_\ve(\varphi_{\ve}(t_0, u) )\leq \sqrt{\delta},
\end{align*}
given that    $J_\ve(u)\leq \inf_{\mathcal{Z}_{\ve}}J_{\ve}+ \delta$ and $\inf_{\mathcal{Z}_{\ve}}J_{\ve} \leq J_\ve(\varphi(t_0,u))$. We have reached  a contradiction, and, therefore, the lemma follows.    
\end{proof}

\begin{proof}[Proof of Theorem \ref{Teo1-Existence}]
  Let $u_{k}$ a minimizing sequences  for $J_{\ve}$ in $\mathcal{Z}_{\ve}$. By Lemma \ref{ekeland},
   we may assume  that $\mathcal{L}_{\ve}( J'_{\ve}(u_{k})) \rightarrow 0$  when $k \rightarrow \infty$.
 From Lemma \ref{JP-S}, $J_{\ve}: H_\ve^{1}(M)\rightarrow \mathbb{R}$ satifies the Palais-Smale condition,
and so there exists a subsequence $u_{k_j}\rightarrow v_\ve$ strongly in $H_\ve ^{1}(M)$ and $J_\ve(v_\ve)=\inf_{\mathcal{Z}_{\ve}}J_{\ve}$.
Since $\mathcal{Z}_{\ve}$ is closed  in $H_\ve ^{1}(M)$, we get that $v_\ve \in \mathcal{Z}_{\ve} $. Finally,
 $\mathcal{Z}_{\ve}$ is invariant
by negative flow, so, $v_\ve$ is  fixed point of flow and, therefore, a solution for  Eq. \eqref{yam-nodal-1}. Now since every sign changins solution  of \eqref{yam-nodal-1} belongs to $\mathcal{E}_\ve$, we have that $v_\ve^{\pm}\in \mathcal{N}_\ve$ and
\[\textbf{d}_\ve =\inf_{\mathcal{Z}_\ve}J_\ve\geq \inf_{\mathcal{E}_\ve}J_\ve \geq 2 \textbf{m}_\ve.\]
\end{proof}

For any $\varepsilon >0$, we  let
\[
E_{\varepsilon} (f) : = \frac{1}{\varepsilon^{n}}  \int_{\R^n} \left(\frac{\varepsilon^2} {2} |\nabla f |^2 + \frac{1} {2} f^2 -\frac{1}{q} \vert f\vert^q  \right) dx.
\]

Now, if we set $U_{\varepsilon} (x)\doteq U\left(\frac{x}{\varepsilon}\right)$, then $U_{\varepsilon}$ is a critical point of $E_{\varepsilon}$, i.e., $U_{\varepsilon}$ is a solution of
\begin{equation}
\label{limequatione}
-\varepsilon^2 \Delta U_{\varepsilon} + U_{\varepsilon} = U_{\varepsilon}^{q-1}.
\end{equation}
Let $x  \in M$, since $ M$ is closed we can fix $r_0 > 0$ such that
$\exp_x\vert_ {B(0,r_0 )} : B(0, r_0) \rightarrow B_g(x, r_0 )$ is a diffeomorphism.
Let $\chi_r$ be a smooth radial cut-off function.
Let us define on M the following function:
\begin{equation}\label{limeqUx}
u_{\ve,x}(y) : =
\begin{cases}
U_{\ve}(\exp^{-1}_x(y))\chi_r(\exp^{-1}_x(y))& \text{if $y\in B_g(x,r)$},\\
0&\text{otherwise}.
\end{cases}
\end{equation}
Now, consider the set $F(M)=\{(x,y)\in M\times M: x\neq y\}$. We define
\begin{equation}
	F_{\ve}(M)  \doteq  \{ (x,  y )\in M\times M :\quad  \mathrm{dist}_{g}(x, y)\geq 2\ve R_0\}\subset F(M),
\end{equation}
where $R_0=\text{diam}(M)$. Moreover, we define the function $i_{\ve}:F_{\ve}(M)\rightarrow H_{\ve}$ by
\begin{equation}
	i_\ve(x,y) \doteq t_{\ve}(u_{\ve,x})u_{\ve,x}-t_{\ve}(u_{\ve,y})u_{\ve, y},
\end{equation}
where $u_{\ve, x}$ and $u_{\ve, y}$ are defined  by \eqref{limeqUx}. Recall that for  $u\in H_{\ve}\setminus\{0\}$, $t_{\ve}(u)u\in \mathcal{N}_{\ve}$ where $t_{\ve}(u)$ is  given by \eqref{tnehari}.

\begin{lemma}
\label{iepc}
	For every $\ve >0$ the function $i_{\ve}:F_{\ve}(M)\rightarrow H_{\ve}$ is continuous. For each $\delta >0$ there exists $\ve_0$ such that, if $\ve\in (0, \ve_0)$ then
	\[
	i_{\ve}(x, y)\in J_{\ve}^{2 \mathbf{m}(E) + \delta}\cap \mathcal{E}_{\ve}\quad \text{for all }\quad  (x, y)\in F_{\ve}(M).
	\]
\end{lemma}
\begin{proof}
Let $x,y\in F_{\ve}(M)$. From Proposition 4.2 in \cite{BBM}, we have that for every $\delta>0$, there exists an $\varepsilon_0>0$ such that for $\varepsilon<\varepsilon_0$, 
\begin{equation}
\label{icont}
t_{\ve}(u_{\ve,x})u_{\ve,x}\quad \text{and}\quad t_{\ve}(u_{\ve,y})u_{\ve, y}\in J^{\mathbf{m}(E)+\frac{\delta}{2}}_{\varepsilon}.
\end{equation}
Observe that
\[
J(i_{\ve}(x,y))=J(t_{\ve}(u_{\ve,x})u_{\ve,x})+J(t_{\ve}(u_{\ve,x})u_{\ve,y}),
\]
given that $u_{\ve,x}$ and $u_{\ve,y}$ have disjoint support. From \eqref{icont} we immediately obtain that $i_{\ve}(x,y)\in J_{\ve}^{2 \mathbf{m}(E) + \delta}$. Finally, using once again that $u_{\ve,x}$ and $u_{\ve,y}$ have disjoint support we get that $i_\ve(x,y)^+= t_{\ve}(u_{\ve,x})u_{\ve,x}\in \mathcal{N}_{\ve}$ and $i_\ve(x,y)^-=-t_{\ve}(u_{\ve,y})u_{\ve, y}\in \mathcal{N}_{\ve}$, therefore, $i_\ve(x,y)\in \mathcal{E}_{\ve}$.
\end{proof}

\begin{proposition} We have that
\label{dep}
 \[ \lim_{\ve\rightarrow 0}\mathbf{d}_{\ve}=2\mathbf{m}(E).\]
\end{proposition}
\begin{proof}
  From Theorem \ref{Teo1-Existence} and Theorem \ref{mep} we have that  
  \[
  \mathbf{d}_\ve \geq 2  \mathbf{m}_\ve \quad \text{ and }\quad  \lim_{\ve\rightarrow 0} \mathbf{m}_\ve=\mathbf{m}(E).
  \] 
  Moreover, from Lemma \ref{iepc}, we get for every $\delta>0$, 
  \[
  d_{\ve}\leq 2\mathbf{m}(E)+\delta,\quad \text{for $\ve>0$ small enough}.
  \]
 Therefore,
\[
\lim_{\ve\rightarrow 0} \mathbf{d}_\ve=2\mathbf{m}(E),
\]
as claimed.
\end{proof}


\section{Concentration of sign changing function in $\mathcal{Z}_{\ve}$}

We begin this section with the following important result.
\begin{lemma}\label{C1}
 Let $u_k\in \mathcal{Z}_{\ve_k}\cap J_{\ve_k}^{\textbf{d}_{\ve_k} + \delta_k}$ where $\ve_k, \delta_k >0$ are such that $\ve_k, \delta_k\rightarrow 0$ as $k\rightarrow \infty$. Then,
  \[
  \mathrm{dist}_{\ve_{k}}(u_{k}^{\pm}, \mathcal{N}_{\ve_k})\rightarrow 0 \quad \text{ and }\quad   J_{\ve_k}(u_{\ve_k}^{\pm})\rightarrow \textbf{m}(E) \text{ as } \quad k\rightarrow \infty.  
  \]
\end{lemma}
\begin{proof}
Observe that 
\begin{align}
\label{Lepuk}
\frac{p-2}{2p}\mathcal{L}_{\varepsilon_k}(u_k,u_k)&=J_{\varepsilon_k}(u_k)-\frac{1}{p} \mathcal{L}_\ve(J'_{\varepsilon_k}(u_k),u_k)\\
\nonumber
&\leq J_{\varepsilon_k}(u_k)+\mathcal{L}_{\varepsilon_k}(J'_{\varepsilon_k}(u_k),J'_{\varepsilon_k}(u_k))^{\frac{1}{2}}\mathcal{L}_{\varepsilon_k}(u_k,u_k)^{\frac{1}{2}}.
\end{align}
From Lemma \ref{ekeland} we may assume that $\mathcal{L}_\ve(J'_{\varepsilon_{k}}(u_k))^{\frac{1}{2}}\rightarrow 0$. Therefore, from this fact and \eqref{Lepuk} we get that $\mathcal{L}_{\varepsilon_k}(u_k,u_k)$ is uniformly bounded. Hence, 
\[
\mathcal{L}_\ve(J'_{\varepsilon_k}(u^{\pm}_k),u^{\pm}_k)=\Big|\mathcal{L}_{\varepsilon_k}(u^{\pm}_k,u^{\pm}_k)-|u^{\pm}_k|^p_{p,\varepsilon_k}\Big|\rightarrow 0.
\]
From this we get that $t_{\varepsilon_{k}}(u^{\pm}_k)$, defined by \eqref{tnehari}, tends to $1$ and, therefore,
\[
\mathrm{dist} _{\ve_{k}}(u_{k}^{\pm}, \mathcal{N}_{\ve_k})\leq \mathcal{L}_{\varepsilon_k}(u^{\pm}_k-t_{\varepsilon_k}(u^{\pm}_k)u^{\pm}_k,u^{\pm}_k-t_{\varepsilon_k}(u^{\pm}_k)u^{\pm}_k)^{\frac{1}{2}}\rightarrow 0, \quad \text{as $k\rightarrow 0$}.
\]
If we use this, together with Theorem \ref{mep} and Proposition \ref{dep}, we get
\begin{align*}
2m(E)& \leq \lim_{k\rightarrow\infty} J_{\varepsilon}(t_{\varepsilon_k}(u^{+}_k)u^{+}_k)+\lim_{k\rightarrow\infty} J_{\varepsilon}(t_{\varepsilon_k}(u^{-}_k)u^{-}_k)\\
&= \lim_{k\rightarrow\infty} J_{\varepsilon}(u^{+}_k)+\lim_{k\rightarrow\infty} J_{\varepsilon}(u^{-}_k)\\
&=\lim_{k\rightarrow\infty}J_{\varepsilon_k}(u_k)\\
&=2m(E).
\end{align*}
Therefore, 
\[
\lim_{k\rightarrow\infty} J_{\varepsilon}(u^{\pm}_k)=m(E).
\]
 \end{proof} 
  \begin{remark}
  \label{BallK}
    On any closed Riemannian manifold  $M$ for any $\ve >0 $ there are  points $x_j\in M$, with $ j=1,\ldots, K_{\ve}$, such that the balls $(B(x_j, \ve))$ are disjoint, and the set is maximal under this condition. It follows that the ball $B(x_j, 2\ve)$ cover $M$. It is easy to construct closed sets $A_j$ such that $B(x_{j},\ve)\subset A_j\subset B(x_j, 2\ve)$ which cover $M$ and only intersect in their boundaries. Moreover, one can see by volume comparison argument that, if $\ve$ is small enough, there exists a constat $K>0$, independent of $\ve$, such that for any point in $M$ can be in at most $K$ of the balls $B(x_j, 3\ve)$.
  \end{remark}

\begin{theorem}\label{Teo_1}
  For any $\eta\in (0,1)$ there exist $\ve_0, \delta_0 > 0$ such that for  any $\ve\in (0,\ve_0)$, $\delta\in (0, \delta_0)$ and $u \in \mathcal{Z}_{\ve}\cap J_{\ve}^{\textbf{d}_\ve+ \delta}  $ there exist $x^{+}=x^{+}(u)$ and $x^-=x^-(u)$ in M such that
  \[
  \int_{B(x^{\pm},r)}\vert u^{\pm}\vert^p d\mu_g \geq \eta \int_{M}\vert u^{\pm}\vert^pd\mu_g.
  \]
  \begin{proof}
    We do the proof only for  $\vert u^{+}\vert^{p}$, the proof for $\vert u^{-}\vert^{p}$ is similar. Assume the theorem is not true. Then there exist $0<\eta < 1$ and sequences $\ve_{k}\rightarrow 0$, $\delta_{k}\rightarrow 0$ and  $u_k \in \mathcal{Z}_{\ve_k}\cap J_{\ve_k}^{\textbf{d}_{\ve_k}+ \delta_k} $  such that for all $x\in M$,
    \[\int_{B(x,r)}\vert u^{+}_k\vert^p d\mu_g < \eta \int_{M}\vert u^{+}_k\vert^pd\mu_g.\]

We first show that there exist $\beta >0 $, $k_0\in\mathbb{N}$  for each $k>k_0$, a point $x_k\in M$ such that
\begin{equation}
\label{concuk}
\dfrac{1}{\ve^n_k}\int_{B(x_k,2 \ve_k )}\vert u^{+}_k \vert^{p} d\mu_{g} > \beta.
\end{equation}

      Let us consider $\ve_0>0$ such that for any $0<\ve< \ve_0$ one can construct a set like in the  Remark \ref{BallK}.  Let $u_{k,j} \doteq  u^{+}_{k}\chi_{A_{j}^k}$ be the restriction of $u^+_{k}$ to $A_j^{k}$ and $0$ away  from $A_j^{k}$. Then,
\begin{align*}
 \vert u^+_k\vert_{p,\ve_k}^{p}&= \dfrac{1}{\ve^n_k}\int_{M}\vert u^{+}_k \vert^{p} d\mu_{g}= \sum_{j}\vert u_{k,j} \vert_{p,\ve_k}^{p}\\
&= \sum_{j}\big(\vert u_{k,j} \vert_{p,\ve_k}^{p-2}\big)\big(\vert u_{k,j} \vert_{p,\ve_k}^{2}\big)\leq \Big( \max_{j}\vert u_{k,j}\vert_{p,\ve_k}^{p-2}\Big)\sum_{j}\vert u_{k,j}\vert _{p,\ve_k}^{2}.
\end{align*}

      Now, let $\varphi_{\ve_k}$ be the cut-off functions on $\mathbb{R}^{n}$ which are $1$ in $B(0, 2\ve_k)$  and vanish away from $B(0, 3\ve_k)$. Moreover, $\|\nabla \varphi_{\ve_k}\|=\dfrac{1}{\ve_{k}}$ in the intermediate annulus.
      Define for $j=1,\ldots,K_\ve$,
\[
\widetilde{u}_{k,j}=u^{+}_k(x)\varphi_{\ve_k}(d(x,x_j)).
\]

     Since $u_{k,j}\leq \widetilde{u}_{k,j}$,  we have $\vert u_{k,j}\vert_{p,\ve_k}^{2}\leq \vert \widetilde{u}_{k,j}\vert_{p,\ve_k}^{2}\leq C\|\widetilde{u}_{k,j}\|_{\ve_k}^{2}.$
      Then, since we have that $\widetilde{u}_{k,j}\leq u^{+}_k$ and on $C_{k,j}= B(x_j, 3\ve_k)-A_j^k$,
      \[\ve^2_k\|\nabla \widetilde{u}_{k,j}\|^2\leq 2\ve^2_k\|\nabla u^+_k\|^2 + 2(u^+_k)^2.\]
      We have  that
      \begin{equation}
      \vert u_{k}^{+}\vert_{p,\ve_k}^p\leq c\|u_k^{+} \|_{\ve_k}^{2}\max_{j}\vert u_{k,j}\vert_{p,\ve_k}^{p-2}.
      \end{equation}
      Now, from Lemma \ref{C1}, 
\begin{equation}\label{C3}
\lim_{k\rightarrow \infty}\| u_{k}^{\pm}\|_{\ve_k}^{2}=\lim_{k\rightarrow \infty}\vert u_{k}^{\pm}\vert_{p,\ve_k}^{p}=\dfrac{2p}{p-2}\textbf{m}(E).
\end{equation}
 Therefore, there exists a $\beta >0 $ such that for each $k$ large enough we can find a $j\in \{1,\ldots,K_{\varepsilon}\}$ such that
      \[
      \beta <\vert u_{k,j}\vert_{p,\ve_k}^{p}=\dfrac{1}{\ve_{k}^{n}}\int_{A_{j}^{k}}\vert u_{k}^{+}\vert^{p} d\mu_{g} \leq \dfrac{1}{\ve_{k}^{n}}\int_{B(x_j, 2\ve_{k})}\vert u_{k}^{+}\vert^{p} d\mu_{g}.
      \]
From this it follows that \eqref{concuk} is established. 

Now, from \eqref{concuk} and given that $t_{\varepsilon_k}(u^{+}_k)$ tends to $1$, there is a $k'_0\in\mathbb{N}$ such that for each $k>k'_0$
\begin{equation}
\label{tukuk}
    \dfrac{1}{\ve^n_k}\int_{B(x_k,2 \ve_k )}\vert t_{\varepsilon_k}(u^{+}_k)u^{+}_k \vert^{p} d\mu_{g} > \frac{\beta}{2}.
\end{equation} 
From Lemma \ref{C1} we get that
\[
m_{\varepsilon_k}\leq J_{\varepsilon_k}(t_{\varepsilon_k}(u^{+}_k)u^{+}_k)\leq m_{\varepsilon_k}+\overline{\delta_k},
\]
for some sequence $\{\overline{\delta_k}\}$ such that $\overline{\delta_k}\rightarrow 0$. 

Set $v_k \doteq t_{\varepsilon_k}(u^{+}_k)u^{+}_k$. Then, for each $k\in\mathbb{N}$ we have that $v_k\in \Sigma_{\varepsilon_k, m_{\varepsilon}+\overline{\delta_k}}$, where
\[
\Sigma_{\varepsilon_k, m_{\varepsilon}+\overline{\delta_k}} \doteq  \{u\in N_{\varepsilon_k}:J_{\varepsilon_k}(u)< m_{\varepsilon}+\overline{\delta_k}\}.
\]
Moreover, from \eqref{tukuk} we have that
\[
\dfrac{1}{\ve^n_k}\int_{B(x_k,2 \ve_k )}\vert v_k \vert^{p} d\mu_{g} > \frac{\beta}{2}.
\]

Now, Lemma 3.4 in \cite{Jimmy} gives a function $\bar{v}_k = v_{k,1}
 + v_{k,2}$ such that $\bar{v}_k \in \Sigma_{\varepsilon_k, m_{\varepsilon}+\overline{\delta_k}}$,  and $v_{k,1}$ is
supported inside a ball centered at $x_k$, $v_{k,1}$ and $v_{k,2}$ have disjoint support and $v_k = \bar{v}_k$ in $B(x_k, 2\ve_k)$ and outside $B(x_k,r)$. 

Then, we have that
\[
\dfrac{1}{\ve^n_k}\int_{M}\vert v^+_{k,1} \vert^{p} d\mu_{g} > \frac{\beta}{2},
\]
and 
\[
\dfrac{1}{\ve^n_k}\int_{M}\vert v^+_{k,2} \vert^{p} d\mu_{g}\geq \dfrac{(1-\eta)}{\ve^n_k}\int_{M}\vert v_{k} \vert^{p} d\mu_{g}\geq (1-\eta)\frac{2p}{p-2}\mathbf{m}_{\ve}.
\]
Here $\eta\in (0,1)$ and it is chosen to be very close to $1$. Now,  from Corollary 3.3 in \cite{Jimmy}, there exists $\delta_0 > 0$, independent of $k$, such that $J_{\ve_k}(\bar{v}_k) \geq \Psi(\delta_0)\mathbf{m}_{\ve}$, where $\Psi:(0,1)\rightarrow (1,\infty)$ is defined in Lemma 3.2 in \cite{Jimmy}.

On the other hand for $k$ large enough we have that $J_{\ve_k}(\bar{v}_k) < \mathbf{m}_{\ve}+ 2\delta_k < \Psi(\delta_0)\mathbf{m}_{\ve}$, reaching a contradiction.

  \end{proof} 
\end{theorem}  

\section{Multiplicity of Nodal Solutions}

Recall that $F(M)=\{(x,y)\in M\times M: x\neq y\}$, and 
\begin{equation}
  F_{\ve}(M) \doteq \{ (x,  y )\in M\times M :\quad  \mathrm{dist}_{g}(x, y)\geq 2\ve r_0\}\subset F(M),
\end{equation}
where $R_0=\text{diam}(M)$. We define the function $i_{\ve}:F_{\ve}\rightarrow H_{\ve}$ by
\begin{equation}
i_\ve(x,y)=t_{\ve}(u_{\ve,x})u_{\ve,x}-t_{\ve}(u_{\ve,y})u_{\ve, y},
\end{equation}
where $t_{\ve}(u)\in \mathbb{R}$  such that if $u\in H_{\ve}-\{0\}$ then $t_{\ve}(u)u\in \mathcal{N}_{\ve}.$

\begin{lemma}
For every $\ve >0$ the function $i_{\ve}$ is continuous. For each $\delta >0$ there exists $\ve_0$ such that, if $\ve\in (0, \ve_0)$ then
\[i_{\ve}(x, y)\in J_{\ve}^{2 c_\infty + \delta}\cap \mathcal{E}_{\ve}\quad \text{for all }\quad  (x, y)\in F_{\ve}(M).\]
\end{lemma}

\subsection{Center of Mass.}

In   \cite{GroveKarcher}, H. Karcher and K. Grove define the center of mass of a function $u$ on a closed Riemannian manifold  $(M,g)$, in the following form, since $M$ is closed the exists $r_0>0$ such that for any $x\in M $ and $r\leq r_0$ the geodesic ball of the radius $r$ center in at $x$, $B(x,r)$ is strongly convex (see \cite{Jimmy} and \cite{GroveKarcher} for details). Let $u\in L^{1}(M)$ nonnegative. We consider the function continuous $P_{u}:M \rightarrow \mathbb{R}$,
\[
P_{u}(x)\doteq\int_{M}(d(x,y))^{2}u(y)d\mu_{g}(y).
\]
Then,  H. Karcher and K. Grove, proved that if $r>0$ is small enough such that the support of $u$ is contained in $B(x,r)$, then $P_{u}$ as a unique global minimum, which they defined as the center of mass of $u$ and denoted by $\textbf{cm}(u).$

We consider now the  center mass of a function introduced in Section 5 of \cite{Jimmy}. For  any function $u\in L^{1}(M)$ and positive $r$ let  the $(u,r)-$\emph{concentration function}    defined by
\begin{equation}
\label{Cur}
  C_{u,r}(x)\doteq\dfrac{\int_{B(x,r)}\vert u\vert d\mu_{g}}{\|u\|_{_{L^{1}(M)}}}.
  \end{equation}

We have that $C_{u,r}\colon  M\rightarrow [0,1]$ it is a continuos function. Where if $r\geq \mathrm{diam}(M)$, then $C_{u,r}\equiv 1$ and $\lim_{r\rightarrow 0} C_{u,r}(x)=0$.

We define the $r-$\emph{concentration  coefficient} of $u$, $C_r(u)$   be the maximum of $C_{u,r}$,
\begin{equation}
  C_{r}(u)\doteq\max_{x\in M}C_{u,r}(x).
  \end{equation}

For any $\eta\in (0,1)$ let $L^{1}_{\ve,\eta}(M,g)\doteq\{u\in L^{1}(M): C_{r}(u)>\eta\}$. We  will use  the  following construction, for any $\eta\in (1/2,1)$ consider the piecewise linear continuous function $\varphi_{\eta}:\mathbb{R}\rightarrow [0,1]$ defined by $\varphi_{\eta}(t)=0$ if $t\leq 1-\eta$  and $\varphi_{\eta}(t)=1$ if $t\geq \eta$  it is a linear and increasing in $[1-\eta, \eta]$.

For  $r>0$ such that $2r\leq r_0$, we let
\[
\Phi_{r,\eta}(u)(x) \doteq \varphi_{\eta}(C_{u,r}(x))u(x),\quad \text{where $u\in L_{r,\eta}^{1}(M) $ and $x\in M$}.
\]

 For the proof of  the following results, namely Lemma \ref{L5.1} and Theorem \ref{Th5.2}, see Pag. 15 of  the already mentioned paper \cite{Jimmy}.

\begin{lemma}
\label{L5.1}
  For any  $u\in L_{r,\eta}^{1}(M) $  the support of $\Phi_{r,\eta}(u)$ is contained in a geodesic ball of radius $2r.$
  (centered at a point of maximal r-concentration)
\end{lemma}
\begin{theorem}
\label{Th5.2}
For any $0<r< 1/2 r_0$  and $\eta > 1/2$  there exists  continuos function $\textbf{Cm}(r,\eta): L_{r,\eta}^{1}\rightarrow M$ , such that if $x\in M$ verifies that $C_{r,u}(x)>\eta$ then $\textbf{Cm}(r,\eta)(u)\in B(x,2r).$ Where \begin{equation}\label{centermass}
  \textbf{Cm}(r,\eta)(u)= \textbf{cm}(\Phi_{r,\eta}(u)).\end{equation}
\end{theorem}

\begin{definition}
  For any function $u$ as in Theorem \ref{Th5.2}, $\textbf{Cm}(r,\eta)(u)$ will be called a $(r,\eta)-$\emph{Riemannian center of mass} of $u$.
  \end{definition}


\begin{proposition}
\label{Cmre}
Let  $0<r< 1/2 r_0$. Then, there exist $\delta_{0} > 0$ and $\varepsilon_{0} > 0$ such that, for any  $u\in  \mathcal{Z}_{\ve}\cap J_{\ve}^{\textbf{d}_\ve + \delta}$  with $\varepsilon \in (0,\varepsilon_0)$ and $\delta \in (0, \delta_0 ]$,
    \[ \textbf{Cm}(\ve r,\eta)((u^+)^p)\neq  \textbf{Cm}(\ve r,\eta)((u^-)^{p}).\]
\end{proposition}

\begin{proof} Let $\ve_k, \delta_k > 0 $ and $u_k \in \mathcal{Z}_{\ve_k}\cap J_{\ve_k}^{\textbf{d}_{\ve} + \delta}$ be such that
  $\ve_k \rightarrow 0,$ $\delta_k \rightarrow 0$ and for each $k$, $\textbf{Cm}(\ve_k r,\eta)((u^+_k)^p) = \textbf{Cm}(\ve_k r,\eta)((u^{-}_k)^{p}).$ From Theorem \ref{Teo_1}, there exist sequences $q^{+}_k, q^{-}_{k}\in M$ such that
\begin{equation}\label{ideabasic1}
  \int_{B(q^{\pm}_{k},\ve_k r)}\vert u^{\pm}_k\vert^p d\mu_g \geq \eta \int_{M}\vert u^{\pm}_k\vert^pd\mu_g.
  \end{equation}
From \eqref{Cur} and \eqref{ideabasic1},
\[  C_{u^+,\ve_k r}(q_k^{+})\doteq\dfrac{\int_{B(q_{k}^{+},\ve_k r)}\vert (u^+)^p\vert d\mu_{g}}{\|(u^+)^p\|_{_{L^{1}(M)}}}\geq \eta.\]
Moreover, if we use Theorem \ref{Th5.2}, we get
  \[ \textbf{Cm}(\ve_k r,\eta)((u^+_k)^p) \in B( q^{+}_k, 2\ve_k r).\]
Hence,  
  \[\lim_{k \rightarrow \infty }\| \textbf{Cm}(\ve_k r,\eta)((u^+_k)^p) - q^{+}_k \| =0. \]
    In similar fashion, we also have
 \[\lim_{k \rightarrow \infty }\| \textbf{Cm}(\ve_k r,\eta)((u^-_k)^p) - q^{-}_k \| =0.  \]
 
  Now, given that   $\textbf{Cm}(\ve_k r,\eta)((u^+_k)^p) = \textbf{Cm}(\ve_k r,\eta)((u^{-}_k)^{p})$ and  $M$ is compact, we have that  
  \begin{equation}
  \label{Pointq}
  q^{+}_k \rightarrow q\quad \text{ and}\quad q^{-}_k \rightarrow q.
  \end{equation}
  
  We now define $w_k^1,w^2_k  \in H^{1}(\mathbb{R}^n)$ by
  \begin{equation}
  \label{wki}
    w_k^{1}(x)\doteq\chi(\ve_k \|x \|)u_k (exp_{c_k^{+}}(\ve_k x))\quad \text{and}\quad w_k^{2}(x)\doteq\chi(\ve_k \|x \|)u_k (exp_{c_k^{-}}(\ve_k x)),
  \end{equation}
  where $c_k^{+}\doteq \textbf{Cm}(\ve_k r,\eta)((u^+_k)^p)$ and $c_k^{-}\doteq \textbf{Cm}(\ve_k r,\eta)((u^- _k)^p.$ Note that, since we are considering the centers of mass $c_k^{\pm}$, then  $ w_k^{1} \neq 0$ and $ w_k^{2} \neq 0$.
  
  By (\ref{C3}), the sequence $\|u_k\|_{\ve_{k}}$ is bounded, so $w_k^{1},w_k^{2}$ are bounded in $ \in H^{1}(\mathbb{R}^n)$ (see Lemma 5.6 in \cite{BBM}).
  Therefore, we have that, up to a subsequence, $w_{k}^{i}\rightharpoonup w^{i} $ weakly in $H^1(\mathbb{R}^n)$, $w_{k}^{i}\rightarrow w^{i}$ a.e in  $\mathbb{R}^n$, and
  $w_{k}^{i}\rightarrow w^{i}$ strongly in $L_{loc}^{p}(\mathbb{R}^n)$ for $i=1,2$.
   
   Now, from Theorem \ref{Teo_1}, $(w^{1})^+ \neq 0$, then $w^1 > 0.$ Analogously $w^2 < 0$. Both $w^{1}$ and $w^2$ are  weak solutions of equation $-\Delta w + w =\vert w\vert^{p-2}w$ and $J_{\infty}(w^i)\leq 2c_{\infty}$.
Since in our setting we still have Ekeland's Lemma, see Lemma \ref{ekeland}, the  proof follows the argument of Lemma 5.7 in \cite{BBM}.
  
  We consider the function
  \begin{equation}
    w_k(x)\doteq\chi(\ve_k x)u_k (exp_{q}(\ve_k x)),\quad \text{where $q$ is as in \eqref{Pointq}}.
  \end{equation}

   Once more, up to a subsequence $w_{k}\rightharpoonup w $ weakly in $H^1(\mathbb{R}^n), $ $w_{k}\rightarrow w$ a.e in  $\mathbb{R}^n$, and
   $w_{k}\rightarrow w$ strongly in $L_{loc}^{p}(\mathbb{R}^n)$, and  $w\neq 0.$ In order to see this, we notice that for every $\varphi \in C_c^{\infty}(\mathbb{R}^n)$  and $k$ large enough, 
\[\int_{\mathbb{R}^n} w_k(x)\varphi(x)dx = \int_{\mathbb{R}^n}w_k^{1}(\psi_k(x))\varphi(x)dx  = \int_{\mathbb{R}^n}w_k^{1}(x)\varphi(\psi_k^{-1}(x))\|\det \psi_{k}^{'}(x)\| dx, \]
where 
\begin{equation*}
\psi_k(x)= \ve_k^{-1}\exp_{c_k^+}^{-1}(\exp_{q}(\ve_k x))
\end{equation*}
Now, if $k\rightarrow \infty $, we have

\[\int_{\mathbb{R}^n} w(x)\varphi(x)dx = \int_{\mathbb{R}^n}w^{1}(x)\varphi(x)dx \qquad \text{for all } \varphi\in C_c^{\infty}(\mathbb{R}^n). \]
So, $w=w^1$. Following a similar argument, we also have $w=w^2$. Therefore, $w=w^1  > 0$ and   $w=w^2 <  0$. This is a contradiction. 
\end{proof}

For $\delta_0>0$ and $\ve_0>0$ as in  Proposition \ref{Cmre}, we define the map $  \textbf{c}_{\ve}:  \mathcal{Z}_{\ve}\cap J_{\ve}^{\textbf{d}_\ve + \delta_0}\rightarrow  F(M)$ by
\begin{equation}
 \textbf{c}_{\ve}(u)\doteq( \textbf{Cm}(r,\eta)((u^+)^p), \textbf{Cm}(r,\eta)((u^-)^{p}))
\end{equation}

\begin{remark} The group $\mathbb{Z}_{2}=\{-1,1\}$ acts in $F(M)$  by $\theta \cdot (x,y)=(y,x)$. This action is free, moreover,
 the maps $\textbf{c}_\ve$ and $i_{\ve}$ are $\mathbb{Z}_{2}-$invariants. In this way we have defined the following
 function \[\widehat{\textbf{c}_{\ve}}: \Big(\mathcal{Z}_{\ve}\cap J_{\ve}^{\mathbf{d}_\ve + \delta_0}\Big)/\mathbb{Z}_{2}\rightarrow C(M)=F(M)/\mathbb{Z}_{2}.\]
\end{remark}

\begin{remark} Let $\check{\mathcal{H}}$ be \v Cech cohomology with $\mathbb{Z}_{2}$ coefficients. This cohomology coincides with singular cohomology $\mathcal{H}^{*}$ on manifolds.
\end{remark}

For $C_{\ve}\doteq F_{\ve}/\mathbb{Z}_{2}$, we have the following result.

\begin{proposition}
  There exists a homomorphism
  \[\tau_{\ve}:\check{\mathcal{H}}\Big( \Big(\mathcal{Z}_{\ve}\cap J_{\ve}^{\mathbf{d}_\ve + \delta_0}\Big)/\mathbb{Z}_{2}\Big )\rightarrow \mathcal{H}^{*}(C_{\ve}(M))\]
  such that  the composition
  \[ \tau_{\ve}\circ \widehat{\textbf{c}_{\ve}}^{*}: \mathcal{H}^{*}(C(M))\rightarrow \mathcal{H}^{*}(C_{\ve}(M))\]
  is the homomorphism induced by the inclusion $C_{\ve}(M)\hookrightarrow C(M)$, wich is an isomorphism for $\ve>0$ small enough.
\end{proposition}

Recall that the cup-length of a topological space $X$, denote it by $\text{cupl}\,(X)$, is the smaller integer $k \geq 1$ such that  the cup-product of any $k$ cohomology
class in $\widetilde{\mathcal{H}}^{*}(X)$ is zero, where $\widetilde{\mathcal{H}}^{*}(X)$ is the reduced cohomology.

\begin{proof}[Proof of Theorem  \ref{Teo-Mult-Cat}]
From Lemma 2.1 we have that  $J_\ve$ satisfies the Palais-Smale condition in $\mathcal{Z_\ve}\cap J_{\ve}^{\mathbf{d}_{\ve} + \delta_0} $. Suppose that contains $k$ pairs $\pm u_1,\ldots,\pm u_k$ critical points of $J_\ve$ and
$J_\ve(u_1)\leq J_\ve(u_2)\leq \cdots \leq J_\ve(u_k).$ From Lemma \ref{ekeland},  we have that
$\mathcal{Z_\ve}\cap J_{\ve}^{\mathbf{d}_{\ve} + \delta_0} $ is positively invariant for the negative gradient flow
 $ \varphi_\ve$ of $\nabla J_\ve$. Hence, for all $u\in\mathcal{Z_\ve}\cap J_{\ve}^{\mathbf{d}_{\ve} + \delta_0} $ there exists
 $j$ with $\varphi_\ve(t,u)\rightarrow \pm u_j$ as $t\rightarrow \infty.$ Let us consider the sets
\[
X_j \doteq \{ u\in\mathcal{Z_\ve}\cap J_{\ve}^{\mathbf{d}_{\ve} + \delta_0}: \varphi_\ve(t,u)\rightarrow \pm u_j \text{ as }  t\rightarrow \infty\}.
\]

The sets $X_j$ are pairwise disjoint  and cover $\mathcal{Z_\ve}\cap J_{\ve}^{\mathbf{d}_{\ve} + \delta_0}$. Now, by the 
Palais-Smale condition for $J_\ve$ in $\mathcal{Z_\ve}\cap J_{\ve}^{\mathbf{d}_{\ve} + \delta_0}$, we have that the union
 $X_1\cup\cdots\cup X_j$ for every $j=1,\ldots, k,$ is an open set of   $\mathcal{Z_\ve}\cap J_{\ve}^{\mathbf{d}_{\ve} + \delta_0}$, therefore $X_j$ is a
 locally closed subset of  $\mathcal{Z_\ve}\cap J_{\ve}^{\mathbf{d}_{\ve} + \delta_0}$. Using the flow $\varphi_\ve$,   $X_j$ can be
 deformated to $\pm u_j$ in  $\mathcal{Z_\ve}\cap J_{\ve}^{\mathbf{d}_{\ve} + \delta_0}$. Hence,
  \[\mathrm{Cat}_{\mathbb{Z}_2}(\mathcal{Z_\ve}\cap J_{\ve}^{\mathbf{d}_{\ve} + \delta_0})\leq k.\]
\end{proof}

\begin{proof}[Proof of Theorem 1.3] We have by Theorem  \ref{Teo-Mult-Cat} that $J_\ve$ has at  least $\mathrm{Cat}_{\mathbb{Z}_2}(\mathcal{Z_\ve}\cap J_{\ve}^{\mathbf{d}_{\ve} + \delta_0})$ sign
changing solutions. Moreover, we have that $\mathrm{Cat}_{\mathbb{Z}_2}(\mathcal{Z_\ve}\cap J_{\ve}^{\mathbf{d}_{\ve} + \delta_0}) \geq  \mathrm{Cupl}\Big(  \Big(\mathcal{Z}_\ve \cap J_{\ve}^{d_{\ve} + \delta_0} \Big)/\mathbb{Z}_2 \Big)\Big)$, see \cite{BMW} sections 5.2. The inclusion $i_\ve : F_{\ve}(M) \hookrightarrow F(M)$ is a homotopy equivalent for all $\ve \in (0,\ve_0)$, therefore, from the Proposition 5.8  it follows that
  \[
   \mathrm{Cupl}((\mathcal{Z_\ve}\cap J_{\ve}^{\mathbf{d}_{\ve} + \delta_0})/\mathbb{Z}_2) \geq  \mathrm{Cupl}\,C(M).
   \]
\end{proof}


\end{document}